\documentclass[a4paper,11pt]{amsart}
  \usepackage{amsmath}
  \usepackage{url,graphicx}
  \usepackage{amssymb, color, pstricks-add, manfnt}
  \usepackage{amsmath, amsthm,  amsfonts}
  \usepackage{mathrsfs}
  \usepackage{cite}

  \theoremstyle{plain}
  \newtheorem{Theorem}{Theorem}[section]
  \newtheorem{Lemma}{Lemma}[section]

  \newtheorem{Corollary}{Corollary}[section]

    \theoremstyle{remark}
  \newtheorem{remark}{Remark}

  \numberwithin{equation}{section}
  \numberwithin{figure}{section}
  \numberwithin{remark}{section}

\begin{document}

\title{Interior $C^{1,1}$ regularity of solutions to degenerate Monge-Amp\`{e}re type equations}

\author{Feida Jiang}
\address{College of Mathematics and Statistics, Nanjing University of Information Science and Technology, Nanjing 210044, P.R. China}
\email{jfd2001@163.com}
\author{JuHua Shi}
\address{School of Science, Nanjing University of Science and Technology, Nanjing 210094, P.R. China}
\email{ashijuhua@163.com}
\author{Xiao-Ping Yang$^*$}
\address{Department of Mathematics, Nanjing University, Nanjing 210093, P.R. China}
\email{xpyang@nju.edu.cn}

\thanks{This work was supported by the National Natural Science Foundation of China (No. 11771214, 11531005).}

\subjclass[2010]{35J96, 35J70, 35J60}

\date{\today}
\thanks{*corresponding author}

\keywords{Degenerate Monge-Amp\`ere type equations;  interior regularity; interior second derivative estimate; Pogorelov type estimate.}

\begin{abstract}
In this paper, we study the interior $C^{1,1}$ regularity of viscosity solutions for a degenerate Monge-Amp\`{e}re type equation $\det[D^{2}u-A(x, u, Du)]=B(x, u, Du)$ when $B \geq 0$ and $B^{\frac{1}{n-1}}\in C^{1,1}(\bar{\Omega}\times\mathbb{R}\times \mathbb{R}^n)$. We prove that $u\in C^{1,1}(\Omega)$ under the A3 condition and A3w$^+$ condition respectively. In the former case, we construct a suitable auxiliary function to obtain uniform {\it a priori} estimates directly. In the latter case,  the main argument is to establish the Pogorelov type estimates, which are interesting independently.

\end{abstract}

\maketitle
\section{ Introduction}

In this paper, we shall study the following degenerate Monge-Amp\`{e}re type equation (DMATE)
\begin{equation}\label{EQUATION}
\det[D^{2}u-A(\cdot, u, Du)]=B(\cdot,u, Du),\quad {\rm in}\ \Omega,
\end{equation}
where $\Omega$ is a bounded domain, $Du$ and $D^2u$ denote the gradient and Hessian matrix of second order derivatives of the unknown function $u: \Omega\rightarrow \mathbb{R}$ respectively, $A: \Omega\times\mathbb{R}\times\mathbb{R}^n\rightarrow \mathbb{R}^{n\times n}$ is a symmetric $n \times n$ matrix valued function and $A\in C^{2,1}(\bar{\Omega}\times\mathbb{R}\times \mathbb{R}^{n}, \mathbb{R}^{n\times n})$, $B: \Omega\times\mathbb{R}\times\mathbb{R}^n\rightarrow \mathbb{R}^+\cup \{0\}$ is a nonnegative scalar function and $B^{\frac{1}{n-1}}\in C^{1,1}(\bar{\Omega}\times\mathbb{R}\times\mathbb{R}^n)$. We shall use $x, z$ and $p$ to denote the points in $\Omega, \mathbb{R}$ and $\mathbb{R}^n$, respectively.

We say that $A$ is  strictly regular in $\Omega$, if
\begin{equation}\label{A3}
\sum_{i,j,k,l=1}^n D^{2}_{p_kp_l}A_{ij}(x, z, p)\xi_i\xi_j\eta_k\eta_l\geq c_0|\xi|^2|\eta|^2,
\end{equation}
holds for all $(x,z,p)\in \Omega \times \mathbb{R} \times \mathbb{R}^n$, $\xi,\eta\in \mathbb{R}^n$ with $\xi\cdot\eta=0$, and some positive constant $c_0$. If $c_0$ on the right hand side in \eqref{A3} is replaced by $0$, we say that $A$ is regular in $\Omega$. As usual, the strictly regular condition and regular condition are also said to be the A3 condition and the A3w condition, respectively, see \cite{MTW, Loeper}. If \eqref{A3} holds for $c_0=0$ without the restriction $\xi \cdot \eta=0$, we call \eqref{A3} the regular condition without orthogonality or the A3w condition without orthogonality. We introduce a particular form of A3w condition, namely
\begin{equation}\label{A3w+}
\sum_{i,j,k,l=1}^n D^{2}_{p_kp_l}A_{ij}(x, z, p)\xi_i\xi_j\eta_k\eta_l\geq \mu_0(\xi \cdot \eta)^2,
\end{equation}
holds for all $(x,z,p)\in \Omega \times \mathbb{R} \times \mathbb{R}^n$, $\xi,\eta\in \mathbb{R}^n$, and some constant $\mu_0$. We call \eqref{A3w+} the A3w$^+$ condition. 
It is obvious that the A3w$^+$ condition implies the A3w condition. The A3w condition without orthogonality implies the A3w$^+$ condition when $\mu_0\le 0$.

The aim of this paper is to investigate interior regularity of solutions to the degenerate equation \eqref{EQUATION}. 
It is well known that the Pogorelov estimate plays an important role in establishing interior regularity of solutions to Monge-Amp\`{e}re equations. 
When $A\equiv 0$, the equation \eqref{EQUATION} reduces to the classical Monge-Amp\`{e}re equation. For the case $B\geq B_0 > 0$ with a constant $B_0$, the Pogorelov estimate for the equation \eqref{EQUATION} together with the homogeneous Dirichlet boundary condition $u=0$ on $\partial\Omega$ was first proved by Pogorelov \cite{Pogorelov}. Various versions of Pogorelov estimates for nondegenerate Monge-Amp\`ere equations can be found in \cite{GT, Figalli, Gutierrez, Wangxujia}. For the case $B > 0$, Blocki \cite{Blocki} proved
\begin{equation}\label{Blocki Pog}
(w-u)^{\alpha}|D^2u|\leq C, \quad {\rm in} \ \Omega,
\end{equation}
where $\alpha=n-1$ if $n \geq 3$ and $\alpha  > 1$ if $n=2$, $w\in C^2(\Omega)$ is convex satisfying $u\le w$ in $\Omega$ and $\lim\limits_{x\rightarrow \partial\Omega}(w(x)-u(x))=0$, and the constant $C$ is independent of the lower bound of $B$.
When $A \not \equiv 0$, the Monge-Amp\`{e}re type equations \eqref{EQUATION} arise in various aspects such as optimal mass transportation problems, geometric optics and conformal geometry etc (see, for instance \cite{MTW, JT-Existence, JT-Pog, Trudinger2006}). The Pogorelov type estimates of non-degenerate Monge-Amp\`{e}re type equations were established under the assumptions of A3w and $A$-boundedness conditions in \cite{LT-Pog, LTW}. Without the $A$-boundedness condition, the interior second order derivative estimates of Pogorelov type were also shown to be valid in \cite{JT-Pog} by constructing a different barrier function with the help of an admissible function. In the optimal mass transportation setting, interior $C^2$ regularity for non-degenerate Monge-Am\`{e}pre type equations was obtained under the A3 condition in \cite{MTW}.

In this paper, we investigate the interior regularity of a viscosity solution $u$ to the degenerate Monge-Amp\`{e}re type equation \eqref{EQUATION}.  By constructing a suitable auxiliary function to directly obtain uniform {\it a priori} estimates of second order derivatives, we first prove that $u \in C^{1,1}(\Omega)$ under the A3 condition. Then we relax the A3 condition to the A3w$^+$ condition, by assuming some suitable additional conditions, we establish the Pogorelov type estimates, which are independently interesting, and further show that the solution $u$ has interior $C^{1,1}$ regularity.

More precisely, we have the following main results.

\begin{Theorem}\label{Th1.1}
Let $u \in C^4(\Omega)\cap C^{1,1}(\bar{\Omega})$ be a solution of the equation \eqref{EQUATION} in a bounded domain $\Omega\subset \mathbb{R}^n$, where $B$ is a positive function and $B^{\frac{1}{n-1}}\in C^{1,1}(\bar{\Omega}\times\mathbb{R}\times \mathbb{R}^n)$. Assume that 
\begin{equation}\label{DppB}
 D_{pp}\tilde B \ge -C_BI,
\end{equation}
for some nonnegative constant $C_B$, where $I$ is the $n\times n$ identity matrix and $\tilde B=\log B$. Assume that $A\in C^2(\bar{\Omega}\times\mathbb{R}\times\mathbb{R}^n, \mathbb{R}^{n\times n})$ is strictly regular. Then, we have
\begin{equation}\label{second order derivatives estimates}
|D^2u(x)|\leq C,
\end{equation}
where $C$ depends on $n$, ${\rm dist}(x,\partial \Omega)$, $\sup\limits_{\Omega}|Du|$, $\|B^{\frac{1}{n-1}}\|_{C^{1,1}}$, $\|A\|_{C^2}$ and $c_0$.
\end{Theorem}

Before stating the next theorem, we first define the viscosity solution of the equation \eqref{EQUATION}. A function $u$ is called a viscosity subsolution (supersolution) of the equation \eqref{EQUATION}, if for any function ${\phi} \in C^2(\Omega)$ such that $u-{\phi}$ has a local maximum (minimum) at some point $x_0 \in \Omega$, there holds
\begin{equation}
\det[D^2{\phi}(x_0)-A(x_0, \phi(x_0), D{\phi}(x_0))] \geq(\leq) B(x_0, \phi(x_0), D{\phi}(x_0)).
\end{equation}
A function $u$ is a viscosity solution of the equation \eqref{EQUATION} if it is both a viscosity subsolution and a viscosity supersolution of the equation \eqref{EQUATION}.

\begin{Theorem}\label{Th1.2}
Let $\Omega$ be a bounded domain in $\mathbb{R}^n$, and $u$ be a viscosity solution of the equation \eqref{EQUATION}. Assume that $A \in C^2(\bar{\Omega}\times\mathbb{R}\times\mathbb{R}^n, \mathbb{R}^{n\times n})$ is strictly regular, $B$ is a nonnegative function, $B^{\frac{1}{n-1}}\in C^{1,1}(\bar{\Omega}\times\mathbb{R}\times \mathbb{R}^n)$ and $B$ satisfies the condition \eqref{DppB}. Then, we have $u \in C^{1,1}(\Omega)$.
\end{Theorem}

Note that the constant $c_0$ in Theorem \ref{Th1.1} is from the strictly regular condition \eqref{A3} of the matrix $A$. The second order derivative estimate \eqref{second order derivatives estimates} depends on $c_0$, which will blow up when $c_0$ tends to $0$. In this sense, Theorems \ref{Th1.1} and \ref{Th1.2} are not valid for the interior second order derivative estimate under the A3w condition.

However, we can still obtain the interior $C^{1,1}$ regularity for the degenerate Monge-Amp\`{e}re type equation \eqref{EQUATION} under the A3w$^+$ condition with the help of suitable barrier functions. 
In order to construct the barrier functions, we can assume either the $A$-boundedness condition or the existence of a strict subsolution. 

First, we introduce the $A$-boundedness condition as in \cite{LT-Pog, Trudinger2006}. We say that the $A$-boundedness condition holds, if there exists a function $\varphi\in C^2(\bar{\Omega})$ satisfying
\begin{equation}\label{A boundedness}
[D_{ij}\varphi-D_{p_k}A_{ij}(x, z, p)D_k\varphi(x)]\xi_i\xi_j\geq |\xi|^2,
\end{equation}
for all $\xi\in \mathbb{R}^n, (x, z, p) \in \Omega\times \mathbb{R} \times \mathbb{R}^n$. 

Next, we introduce the definition of a strict subsolution of the equation \eqref{EQUATION}. A function $u \in C^{2}(\Omega)$ is called an elliptic (a degenerate elliptic) function when its augmented Hessian matrix $M[u]:=D^{2}u-A(x, u, Du) > 0 (\geq 0)$. 
If $u$ is also a solution of the equation \eqref{EQUATION}, we call it an elliptic (a degenerate elliptic) solution.
A function $\underline u\in C^2(\Omega)$ is said to be elliptic (degenerate elliptic) with respect to $u$ in $\Omega$, if $M_u[\underline u] := D^2 \underline u - A(\cdot, u, D\underline u)>0 (\geq 0)$ in $\Omega$. If such a function $\underline u$ also satisfies
    \begin{equation}\label{sub sol}
    \det(M_u[\underline u]) > B(\cdot,u, D\underline u),
    \end{equation}
at points in $\Omega$, we call $\underline u$ a strict subsolution of the equation \eqref{EQUATION}.

We now formulate the Pogorelov type estimate under A3w$^+$ in the following theorem.

\begin{Theorem}\label{Th1.3}
Let $u \in C^4(\Omega)\cap C^{1,1}(\bar{\Omega})$ be a solution of the equation \eqref{EQUATION} in a bounded domain $\Omega\subset\mathbb{R}^n$, where $B$ is a positive function, $B^{\frac{1}{n-1}}\in C^{1,1}(\bar{\Omega}\times\mathbb{R}\times\mathbb{R}^n)$ and $B$ satisfies the condition \eqref{DppB}. Assume that $A\in C^2(\bar{\Omega}\times\mathbb{R}\times\mathbb{R}^n, \mathbb{R}^{n\times n})$ satisfies the A3w$^+$ condition, and there exists a $C^{1,1}$ function $w$ satisfying $w\ge u$ in $\Omega$, $w=u$ on $\partial\Omega$, which is degenerate elliptic with respect to $u$ in $\Omega$. Assume also one of the following conditions:
\begin{itemize}
\item[(i)] $A$-boundedness condition \eqref{A boundedness} holds;
\item[(ii)] there exists a strict subsolution $\underline u\in C^2(\Omega)$ of the equation \eqref{EQUATION} satisfying \eqref{sub sol}.
\end{itemize}
Then we have the estimate
\begin{equation}\label{pOGORELOV ESTIMATES}
(w-u)^\tau |D^2u| \leq C, \quad {in} \ \Omega,
\end{equation}
where $\tau=2$ if $B_p \not\equiv 0$ and $\tau=1$ if $B_p \equiv 0$, the constant $C$ depends on $n$, $\Omega$, $\|B^{\frac{1}{n-1}}\|_{C^{1,1}}$, $\|A\|_{C^2}$, $\sup\limits_{\Omega}|Dw|$, $\sup\limits_{\Omega}|Du|$. In case (ii), the constant $C$ depends in addition on $\underline u$.
\end{Theorem}

There is a technical reason why we restrict our attention under the A3w$^+$ condition, see Remark \ref{Rm4.1} after the proof of Theorem \ref{Th1.3}.

\begin{remark}
We remark that, in Theorem \ref{Th1.3}, if $B$ satisfies a further condition $\frac{|B_p|}{B} \le C$ for some nonnegative constant $C$, then the estimate \eqref{pOGORELOV ESTIMATES} can be improved to $(w-u) |D^2u| \leq C$, which corresponds to the estimate \eqref{pOGORELOV ESTIMATES} for the $B_p\equiv 0$ case as well.
\end{remark}

From Theorem \ref{Th1.3}, we can have the following interior regularity result.

\begin{Theorem}\label{Th1.4}
Under the assumptions of Theorem \ref{Th1.3}, assume instead that $u$ is a viscosity solution of the equation \eqref{EQUATION} and $B$ is a nonnegative function, and assume further that $A$ and $B$ are nondecreasing in $z$. Then we have $u \in C^{1,1}(\Omega)$.
\end{Theorem}

In order to guarantee the comparison principle, the monotonicity conditions for both $A$ and $B$ with respect to $z$ are assumed in Theorem \ref{Th1.4}.

\begin{remark}
We emphasize that the constants $C$ in both the estimates \eqref{second order derivatives estimates} in Theorem \ref{Th1.1} and \eqref{pOGORELOV ESTIMATES} in Theorem \ref{Th1.3} are independent of the positive lower bound of $B$, so that they can be applied to obtain the interior $C^{1,1}$ regularity for the degenerate equation \eqref{EQUATION}. The assumption $B^{\frac{1}{n-1}}\in C^{1,1}$ can be found in \cite{Guan, GTW}, which is proved to be optimal in \cite{W} when $A\equiv 0$ and $B$ is independent of $z$ and $p$. When $\mu_0 \le 0$, the matrix $A\equiv 0$ satisfies the A3w$^+$ condition \eqref{A3w+} automatically, so that Theorem \ref{Th1.3} and \ref{Th1.4} can apply to the standard Monge-Amp\`ere equation $\det D^2u=B(\cdot, u, Du)$.
\end{remark}

The organization of this paper is as follows. In Section \ref{Section 2}, we introduce some properties of $B$ when $B^{\frac{1}{n-1}}\in C^{1,1}$, in Lemma \ref{Lemma 1} and Corollary \ref{Coro2.1}, which are useful in deriving estimates independent of the lower bound of $B$. A fundamental barrier construction under the A3w condition is also introduced in Lemma \ref{Lemma 2}, which will be used in Section \ref{Section 4} when we only assume the A3w$^+$ condition. In Section \ref{Section 3}, we obtain interior second order derivative estimates for the Monge-Amp\`{e}re type equation \eqref{EQUATION} under A3 condition, and then show the interior $C^{1,1}$ regularity for viscosity solutions of the DMATE \eqref{EQUATION}. In Section \ref{Section 4}, under the A3w$^+$ condition, we establish the Pogorelov type estimates for the Monge-Amp\`{e}re type equation \eqref{EQUATION} by using suitable barrier functions, and apply these estimates to obtain interior $C^{1,1}$ regularity for viscosity solutions of the DMATE \eqref{EQUATION}.

\vspace{3mm}

\section{Preliminaries}\label{Section 2}

In this section, we introduce some properties of $B$ when $B^{\frac{1}{n-1}}\in C^{1,1}(\bar{\Omega}\times\mathbb{R}\times\mathbb{R}^n)$, and a fundamental  lemma of barrier construction, which will be used in later sections.

In the equation \eqref{EQUATION}, we suppose $B>0$ in $\Omega$, $\tilde{u}_{ij}: =u_{ij}-A_{ij}$ and $\{\tilde{u}^{ij}\}:= \{\tilde{u}_{ij}\}^{-1}$. Then both matrices $\{\tilde{u}_{ij}\}$ and $\{\tilde{u}^{ij}\}$ are positive definite.
We can rewrite the equation \eqref{EQUATION} in the form
\begin{equation}\label{EQUATION'}
\log\det \{\tilde{u}_{ij}\}=\tilde{B}, \quad {\rm in} \ \Omega,
\end{equation}
where $\tilde B:= \log B$.
By differentiating the equation \eqref{EQUATION'} in the direction $\xi\in \mathbb{R}^n$ once and twice respectively, we have
\begin{equation}\label{Differentiation of the equation once}
\tilde{u}^{ij}[D_{\xi}u_{ij}-D_\xi A_{ij}-(D_zA_{ij})D_\xi u-(D_{p_k}A_{ij})D_\xi u_k]=D_\xi\tilde{B},
\end{equation}
and
\begin{equation}\label{Differentiation of the equation twice}
\begin{split}
  &\tilde{u}^{ij}[D_{\xi\xi}u_{ij}-D_{\xi\xi}A_{ij}-(D_{p_kp_l}A_{ij})D_\xi u_k D_\xi u_l-(D_{p_k}A_{ij})D_{\xi\xi}u_k -(D_zA_{ij})D_{\xi\xi}u\\
  & -(D_{zz}A_{ij})(D_\xi u)^2 -2(D_{\xi z}A_{ij})D_\xi u -2(D_{zp_k}A_{ij})D_\xi uD_\xi u_k-2(D_{\xi p_k}A_{ij})D_\xi u_k] \\
=&\tilde{u}^{is}\tilde{u}^{jt}D_\xi\tilde{u}_{ij}D_\xi\tilde{u}_{st} +D_{\xi\xi}\tilde{B},
\end{split}
\end{equation}
where
\begin{equation}\label{DiB}
D_\xi\tilde{B}=\frac{B_\xi+B_zD_\xi u+B_{p_k}D_{\xi}u_k}{B},
\end{equation}
and
\begin{equation}\label{DijB}
\begin{split}
D_{\xi\xi}\tilde{B}= &\frac{B_{\xi\xi}+B_{zz}(D_\xi u)^2+B_z D_{\xi\xi}u+B_{p_k}D_{\xi\xi}u_k+B_{p_kp_l}(D_{\xi}u_l)(D_{\xi}u_k)}{B}\\
& +\frac{2B_{\xi z}D_\xi u+2B_{\xi p_l}D_{\xi}u_l+2B_{zp_k}(D_{\xi}u)(D_{\xi}u_k)}{B}\\
&-\frac{B_{\xi}^2+B_z^2(D_{\xi}u)^2+B_{p_k}B_{p_l}(D_{\xi}u_k)(D_{\xi}u_l)}{B^2} \\
&-\frac{2B_{\xi}B_zu_{\xi}+2B_{\xi}B_{p_l}D_{\xi}u_l+2B_zB_{p_l}(D_{\xi}u)(D_{\xi}u_l)}{B^2}.
\end{split}
\end{equation}
Note that we use the standard summation convention in the context that repeated indices indicate summation from $1$ to $n$ unless otherwise specified.

We introduce the following lemma and its corollary, in order to deal with the right-hand side term of the equation \eqref{EQUATION}. 

\begin{Lemma}\label{Lemma 1}
Assume $B^{\frac{1}{n-1}}(x, u, Du)\in C^{1,1}(\bar{\Omega}\times\mathbb{R}\times\mathbb{R}^n)$ and $B > 0$, then we have
\begin{equation}\label{B_i over B}
\left |\frac{B_i}{B}\right |, \left |\frac{B_z}{B}\right |, \left |\frac{B_{p_i}}{B}\right | \leq (n-1) \sqrt{2 \|B^{\frac{1}{n-1}} \|_{C^{1,1}(\bar{\Omega}\times\mathbb{R}\times\mathbb{R}^n)}}B^{-\frac{1}{2(n-1)}},
\end{equation}
in $\bar\Omega\times\mathbb{R}\times\mathbb{R}^n$, for $i=1,\cdots, n$, and
\begin{equation}\label{Bij over B}
\left |\frac{B_{ij}}{B}\right |, \left |\frac{B_{iz}}{B} \right |, \left |\frac{B_{ip_j}}{B} \right |, \left |\frac{B_{zz}}{B}\right |, \left |\frac{B_{zp_i}}{B} \right |, \left |\frac{B_{p_ip_j}}{B} \right | \le (n-1) (2n-3)  \|B^{\frac{1}{n-1}} \|_{C^{1,1}(\bar\Omega\times\mathbb{R}\times\mathbb{R}^n)} B^{-\frac{1}{n-1}},
\end{equation}
in $\bar\Omega\times\mathbb{R}\times\mathbb{R}^n$, for $i,j =1,\cdots, n$.
\end{Lemma}

\begin{proof}
By Taylor's formula, for any given $(x_0, z_0, p_0)\in \bar\Omega \times \mathbb{R} \times \mathbb{R}^n$,
\begin{equation}\label{Taylor}
\begin{split}
0 \le & B^{\frac{1}{n-1}}(x, z, p)\\
\le  & B^{\frac{1}{n-1}}(x_0, z_0, p_0)+\nabla \left (B^{\frac{1}{n-1}} \right )(x_0,z_0, p_0)\cdot \left (x-x_0,z-z_0, p-p_0 \right )\\
&+\frac{1}{2} \|B^{\frac{1}{n-1}}  \|_{C^{1,1}(\bar{\Omega}\times\mathbb{R}\times\mathbb{R}^n)}\left [|x-x_0|^2+|z-z_0|^2+|p-p_0|^2 \right ],
\end{split}
\end{equation}
holds for any $(x, z, p)\in \bar\Omega \times \mathbb{R} \times \mathbb{R}^n$, where $\nabla := (D_x, D_z, D_p)$. Kirszbraun's Theorem (in Section 12.10.43 in \cite{Federer}) asserts that there exists an extension from $\bar \Omega \times \mathbb{R} \times \mathbb{R}^n$ to $\mathbb{R}^n \times \mathbb{R} \times \mathbb{R}^n$ such that $B^{\frac{1}{n-1}}\in C^{1,1}(\mathbb{R}^n\times\mathbb{R}\times\mathbb{R}^n)$ and $\|B^{\frac{1}{n-1}}\|_{C^{1,1}(\mathbb{R}^n\times\mathbb{R}\times\mathbb{R}^n)}=\|B^{\frac{1}{n-1}}\|_{C^{1,1}(\bar{\Omega}\times\mathbb{R}\times\mathbb{R}^n)}$, then \eqref{Taylor} holds for all $(x, z, p)\in \mathbb{R}^n \times \mathbb{R} \times \mathbb{R}^n$. Consequently, we have
\begin{equation}
\left ((B^{\frac{1}{n-1}})_{i}(x_0,z_0, p_0)\right)^2-2 \|B^{\frac{1}{n-1}}\|_{C^{1,1}(\bar{\Omega}\times\mathbb{R}\times\mathbb{R}^n)} B^{\frac{1}{n-1}}(x_0,z_0, p_0)\leq 0, \quad {\rm for} \ i=1,\cdots, n,
\end{equation}
\begin{equation}\label{}
\left ((B^{\frac{1}{n-1}})_{z}(x_0,z_0, p_0)\right)^2-2 \|B^{\frac{1}{n-1}}\|_{C^{1,1}(\bar{\Omega}\times\mathbb{R}\times\mathbb{R}^n)} B^{\frac{1}{n-1}}(x_0,z_0, p_0)\leq 0,
\end{equation}
and 
\begin{equation}\label{}
\left ((B^{\frac{1}{n-1}})_{p_i}(x_0,z_0, p_0)\right )^2-2 \|B^{\frac{1}{n-1}}\|_{C^{1,1}(\bar{\Omega}\times\mathbb{R}\times\mathbb{R}^n)} B^{\frac{1}{n-1}}(x_0,z_0, p_0)\leq 0,
\quad {\rm for} \ i=1,\cdots, n,
\end{equation}
namely,
\begin{equation}\label{Bi}
\left |(B^{\frac{1}{n-1}})_{i}(x_0,z_0, p_0)\right | \leq \sqrt{2 \|B^{\frac{1}{n-1}}\|_{C^{1,1}(\bar{\Omega}\times\mathbb{R}\times\mathbb{R}^n)}}B^{\frac{1}{2(n-1)}}(x_0,z_0, p_0), \quad {\rm for} \ i=1,\cdots, n,
\end{equation}
\begin{equation}\label{Bz}
\left |(B^{\frac{1}{n-1}})_{z}(x_0,z_0, p_0)\right | \leq \sqrt{2 \|B^{\frac{1}{n-1}}\|_{C^{1,1}(\bar{\Omega}\times\mathbb{R}\times\mathbb{R}^n)}}B^{\frac{1}{2(n-1)}}(x_0, z_0, p_0),
\end{equation}
and
\begin{equation}\label{Bpi}
\left |(B^{\frac{1}{n-1}})_{p_i}(x_0,z_0, p_0) \right | \leq \sqrt{2 \|B^{\frac{1}{n-1}}\|_{C^{1,1}(\bar{\Omega}\times\mathbb{R}\times\mathbb{R}^n)}}B^{\frac{1}{2(n-1)}}(x_0, z_0, p_0),
\quad {\rm for} \ i=1,\cdots, n.
\end{equation}
By \eqref{Bi}, \eqref{Bz} and \eqref{Bpi}, we have
\begin{equation}\label{Bxi over B}
\left |\frac{B_i}{B}(x_0,z_0, p_0)\right | \leq (n-1) \sqrt{2 \|B^{\frac{1}{n-1}}\|_{C^{1,1}(\bar{\Omega}\times\mathbb{R}\times\mathbb{R}^n)}} B^{-\frac{1}{2(n-1)}}(x_0,z_0, p_0), \quad {\rm for} \ i=1,\cdots, n,
\end{equation}
\begin{equation}\label{Bz over B}
\left |\frac{B_z}{B}(x_0,z_0, p_0)\right | \leq (n-1) \sqrt{2 \|B^{\frac{1}{n-1}}\|_{C^{1,1}(\bar{\Omega}\times\mathbb{R}\times\mathbb{R}^n)}} B^{-\frac{1}{2(n-1)}}(x_0,z_0, p_0),
\end{equation}
and
\begin{equation}\label{Bpi over B}
\left |\frac{B_{p_i}}{B}(x_0,z_0, p_0)\right | \leq (n-1) \sqrt{2 \|B^{\frac{1}{n-1}}\|_{C^{1,1}(\bar{\Omega}\times\mathbb{R}\times\mathbb{R}^n)}} B^{-\frac{1}{2(n-1)}}(x_0,z_0, p_0), \quad {\rm for} \ i=1,\cdots, n.
\end{equation}
Since $(x_0, z_0, p_0)$ can be an arbitrary point in $\bar \Omega\times\mathbb{R}\times\mathbb{R}^n$, from \eqref{Bxi over B}, \eqref{Bz over B} and \eqref{Bpi over B}, conclusion \eqref{B_i over B} is proved.

Next, by a direct computation, we obtain
\begin{equation}\label{Bij B n-1}
D_{ij}\left (B^{\frac{1}{n-1}}\right )=\frac{1}{n-1}B^{\frac{1}{n-1}} \left (\frac{B_{ij}}{B} - \frac{n-2}{n-1}\frac{B_i}{B}\frac{B_j}{B} \right ),
\end{equation}
in $\bar\Omega\times\mathbb{R}\times\mathbb{R}^n$, for $i,j=1,\cdots, n$.
Therefore, we have from \eqref{Bij B n-1} that
\begin{equation}\label{Bij over B+}
\begin{array}{rl}
\displaystyle \left |\frac{B_{ij}}{B}\right | \!\!&\!\!\displaystyle \le (n-1) \left |(B^{\frac{1}{n-1}})_{ij}\right | B^{-\frac{1}{n-1}} + \frac{n-2}{n-1} \left |\frac{B_i}{B}\right | \left |\frac{B_j}{B}\right | \\
                         \!\!&\!\!\displaystyle \le (n-1) (2n-3) \|B^{\frac{1}{n-1}}\|_{C^{1,1}(\bar\Omega\times\mathbb{R}\times\mathbb{R}^n)} B^{-\frac{1}{n-1}},
\end{array}
\end{equation}
in $\bar\Omega\times\mathbb{R}\times\mathbb{R}^n$, where \eqref{B_i over B} is used in the last inequality. Then \eqref{Bij over B+} completes the proof of the first inequality in \eqref{Bij over B}. The other inequalities in \eqref{Bij over B} can be derived similarly to \eqref{Bij over B+}. We omit the remaining proof, in order to avoid too many repetitions.
\end{proof}

\begin{remark}\label{Rm2.1}
In fact, we can have a relaxed version of the estimate \eqref{B_i over B},
\begin{equation}\label{B_i over B'}
\left |\frac{B_i}{B}\right |, \left |\frac{B_z}{B}\right |, \left |\frac{B_{p_i}}{B}\right | \leq (n-1) \|B^{\frac{1}{n-1}} \|_{C^{0,1}(\bar{\Omega}\times\mathbb{R}\times\mathbb{R}^n)} B^{-\frac{1}{n-1}},
\end{equation}
in $\bar\Omega\times\mathbb{R}\times\mathbb{R}^n$, for $i=1,\cdots, n$, which can be readily verified by a direct calculation. Namely, we have
\begin{equation}
\left |\frac{B_i}{B}\right | =  (n-1) \left | D_i(B^{\frac{1}{n-1}}) \right | B^{-\frac{1}{n-1}}  \leq (n-1) \|B^{\frac{1}{n-1}} \|_{C^{0,1}(\bar{\Omega}\times\mathbb{R}\times\mathbb{R}^n)}B^{-\frac{1}{n-1}},
\end{equation}
for $i=1,\cdots, n$. The estimates for $\left |\frac{B_z}{B}\right |$ and $\left |\frac{B_{p}}{B}\right |$ can be obtained exactly in the same way.
\end{remark}

We have the following consequence of Lemma \ref{Lemma 1} and Remark \ref{Rm2.1}.
\begin{Corollary}\label{Coro2.1}
Assume $B^{\frac{1}{n-1}}(x, u, Du)\in C^{1,1}(\bar{\Omega}\times\mathbb{R}\times\mathbb{R}^n)$, $B > 0$ and $\tilde{B}=\log B$. Then we have the following properties:

(i) \begin{equation}\label{D_itildeB}
|D_i\tilde{B}| \le C\left [1+\max_{j}\left(|\tilde{u}_{ij}|\right )\right ] B^{-\frac{1}{n-1}}
\end{equation}
holds for $i=1,\cdots, n$, where the constant $C$ depends on $n, \|B^{\frac{1}{n-1}} \|_{C^{0,1}}, A$ and $\sup\limits_{\Omega}|Du|$.

(ii) If the condition \eqref{DppB} holds, then
\begin{equation}\label{D_ijtildeB}
D_{ii}\tilde{B} \ge -C\left [1+\max_j\left(|\tilde{u}_{ij}|\right)\right ] B^{-\frac{1}{n-1}}-C^\prime\left [1+\max_j\left(|\tilde{u}_{ij}|\right)\right ]^2 + \sum_{k=1}^n \tilde B_{p_k}D_{ii}u_k
\end{equation}
holds for $i=1,\cdots, n$, where the constant $C$ depends on $n, \|B^{\frac{1}{n-1}} \|_{C^{1,1}}, A$ and $\sup\limits_{\Omega}|Du|$,  and the constant $C^\prime$ depends on  $C_B$ and $A$.
\end{Corollary}

\begin{proof}
Choosing $\xi=e_i$ in \eqref{DiB}, we have, for $i=1,\cdots, n$, 
\begin{equation}\label{Coro1}
D_i\tilde{B}=\frac{B_i+B_zD_i u+B_{p_k}D_{i}u_k}{B}.
\end{equation}
It follows from \eqref{B_i over B'} that 
\begin{equation}\label{Coro2}
\left |\frac{B_i+B_zD_i u}{B}\right | \leq C\left (\left |\frac{B_i}{B}\right |+\left |\frac{B_z}{B}\right |\right)
\leq CB^{-\frac{1}{n-1}},
\end{equation}
where the constant $C$ depends on $n, \|B^{\frac{1}{n-1}} \|_{C^{0,1}}$ and $\sup\limits_{\Omega}|Du|$. Since $\tilde{u}_{ij}=u_{ij}-A_{ij}$, we obtain
\begin{equation}\label{Coro3}
\left |\frac{B_{p_k}D_{i}u_k}{B}\right |
\leq C\left [1+\max_{j}\left(|\tilde{u}_{ij}|\right )\right ] B^{-\frac{1}{n-1}},
\end{equation}
where the constant $C$ depends on $n, \|B^{\frac{1}{n-1}} \|_{C^{0,1}}$ and $A$. Combining \eqref{Coro1}, \eqref{Coro2} and \eqref{Coro3}, we get \eqref{D_itildeB} and finish the proof of conclusion (i).

Next, we turn to prove (ii). It follows from \eqref{B_i over B} and \eqref{Bij over B} that, for $i=1,\cdots, n$, 
\begin{equation}\label{Coro4}
\begin{split}
     & \left|\frac{B_{ii}+B_{zz}(D_i u)^2+B_z D_{ii}u +2B_{iz}D_i u+2B_{i p_l}D_{i}u_l+2B_{zp_k}D_{i}uD_{i}u_k}{B}\right |\\
\le & C\left [1+\max_j\left(|{u}_{ij}|\right)\right ] B^{-\frac{1}{n-1}} \le C\left [1+\max_j\left(|{\tilde u}_{ij}|\right)\right ] B^{-\frac{1}{n-1}},
\end{split}
\end{equation}
and
\begin{equation}\label{Coro5}
\begin{split}
    & \left |\frac{B_{i}^2+B_z^2u_{i}^2+2B_{i}B_zu_{i}+2B_{i}B_{p_l}D_{i}u_l+2B_zB_{p_l}u_{i}D_{i}u_l}{B^2}\right | \\
\le & C\left [1+\max_j\left(|{u}_{ij}|\right)\right ] B^{-\frac{1}{n-1}} \le C\left [1+\max_j\left(|{\tilde u}_{ij}|\right)\right ] B^{-\frac{1}{n-1}},
\end{split}
\end{equation}
where the constants $C$ depend on $n, \|B^{\frac{1}{n-1}} \|_{C^{1,1}}, A$ and $\sup\limits_{\Omega}|Du|$.
By the condition \eqref{DppB}, we have 
\begin{equation}\label{Coro6}
\begin{split}
\frac{B_{p_kp_l}B-B_{p_k}B_{p_l}}{B^2}u_{il}u_{ik}&\ge - C_B \delta_{kl} (\tilde u_{il}+A_{il})(\tilde u_{ik}+A_{ik})\\
&\geq -C^\prime\left [1+\max_j\left(|\tilde{u}_{ij}|\right)\right ]^2,
\end{split}
\end{equation}
where $\delta_{kl}$ denotes the usual Kronecker delta, the constant $C^\prime$ depends on  $C_B$ and $A$.
Taking $\xi=e_1$ in \eqref{DijB}, and using \eqref{Coro4}, \eqref{Coro5} and \eqref{Coro6}, we get \eqref{D_ijtildeB} and finish the proof of conclusion (ii).
\end{proof}
\begin{remark}
We remark that $\tilde B=\log B$ satisfies the condition \eqref{DppB}, if it is semi-convex in $p$.
The term $\sum_{k=1}^n \tilde B_{p_k}D_{ii}u_k$ on the right hand side of \eqref{D_ijtildeB} can also be dealt with in the later discussion.
\end{remark}

By the equation \eqref{EQUATION}, we can build the relationship between $B^{-\frac{1}{n-1}}$ and $\sum_{i=1}^n\tilde{u}^{ii}$, $(\tilde{u}^{ii})=(\tilde{u}_{ii})^{-1}$, if $\tilde{u}_{11} >1$. Therefore, a suitable barrier function is necessary to control the term $C\sum_{i=1}^n\tilde{u}^{ii}$. We introduce the following barrier construction lemma under the A3w condition, which is a variant of Lemma 2.1(ii) in \cite{JT-oblique-II} when the operator $\mathcal{F}$ is given by ``$\log \det$''. Similar versions of such a lemma can also be found in \cite{JT-Pog, JTY}.

\begin{Lemma}\label{Lemma 2}
Let $u\in C^2(\bar\Omega)$ be an elliptic solution of the equation \eqref{EQUATION} and $\underline u\in C^2(\bar \Omega)$ be a strict subsolution of the equation \eqref{EQUATION} satisfying \eqref{sub sol}.
Assume that $A\in C^2(\bar{\Omega}\times \mathbb{R}\times\mathbb{R}^n)$ satisfies the A3w condition, $B\in C^2(\bar{\Omega}\times \mathbb{R}\times\mathbb{R}^n)$ is a positive function satisfying \eqref{DppB}. Then the inequality
\begin{equation}\label{barrier inequality}
L\left [e^{\kappa (\underline u-u)}\right ]\geq \varepsilon_1\sum_{i=1}^n\tilde{u}^{ii}-C,
\end{equation}
holds in $\Omega$ for sufficiently large positive constant $\kappa$ and uniform positive constants $\varepsilon_1$ and $C$, where
\begin{equation}\label{linearized op}
L=\sum_{i, j=1}^{n}\tilde{u}^{ij}\left (D_{ij}-\sum_{k=1}^nD_{p_k}A_{ij}(x, u, Du)D_k\right )-\sum_{k=1}^{n}\tilde B_{p_k}D_k.
\end{equation}
\end{Lemma}
\begin{proof}
Since $\underline u$ is a strict subsolution satisfying \eqref{sub sol}, by taking $\mathcal{F}=\log \det$ in Lemma 2.1(ii) in \cite{JT-oblique-II}, following (2.17) in \cite{JT-oblique-II} we have
\begin{equation}\label{LM2.2(1)}
L \left [e^{\kappa (\underline u-u)}\right ] \ge  \varepsilon_1\left (\sum_{i=1}^n\tilde{u}^{ii} + 1\right ) + C\left [\tilde B(\cdot, u, D\underline u)-\tilde B(\cdot,u, Du)-\sum_{k=1}^n\tilde B_{p_k}(\cdot, u, Du)D_k(\underline u-u)\right ],
\end{equation}
for large positive constant $\kappa$ and uniform positive constant $\varepsilon_1$. By Taylor's formula and the condition \eqref{DppB}, we have
\begin{equation}\label{LM2.2(2)}
\begin{array}{rl}
  & \displaystyle \tilde B(\cdot, u, D\underline u)-\tilde B(\cdot,u, Du)-\sum_{k=1}^n\tilde B_{p_k}(\cdot, u, Du)D_k(\underline u-u) \\
=& \displaystyle \frac{1}{2} \sum_{k,l=1}^n \tilde B_{p_kp_l} (\cdot, u, \hat p) D_k(\underline u-u)D_l(\underline u-u) \\
\ge & \displaystyle - \frac{1}{2}C_B |D(\underline u-u)|^2, 
\end{array}
\end{equation}
where $\hat p=\theta Du + (1-\theta)D\underline u$ with $\theta \in (0,1)$. Then the estimate \eqref{barrier inequality} can be obtained by combining \eqref{LM2.2(1)} and \eqref{LM2.2(2)}.
\end{proof}
In Lemma \ref{Lemma 2}, if the A3w condition holds without orthogonality,  the inequality {barrier inequality} still holds by replacing the barrier function $e^{\kappa (\underline u-u)}$ with $\kappa(\underline u-u)$.
Note also that if $C_B=0$ in condition \eqref{DppB}, namely $\tilde B$ is convex in $p$, then the barrier inequality \eqref{barrier inequality} can be replaced by
\begin{equation}
L \left [e^{\kappa (\underline u-u)}\right ] \ge  \varepsilon_1\left (\sum_{i=1}^n\tilde{u}^{ii} + 1\right ),
\end{equation}
since the second term on the right hand side of \eqref{LM2.2(1)} is nonnegative in this case. 

\vspace{3mm}

\section{Interior regularity for the DMATE \eqref{EQUATION} under the A3 condition}\label{Section 3}
In this section, by constructing an auxiliary function, we obtain interior second order derivative estimates for the Monge-Amp\`ere type equation \eqref{EQUATION} under the A3 condition and $B>0$. We then use the estimates to obtain the interior regularity for the solution of the DMATE \eqref{EQUATION}.

\begin{proof}[Proof of Theorem \ref{Th1.1}]
We employ the auxiliary function
\begin{equation}\label{G}
G(x,\xi)=\eta^2(x)\tilde{u}_{\xi\xi},
\end{equation}
where $\eta$ is a cut-off function in $\Omega$, $0 \leq \eta \leq 1$, $\tilde{u}_{\xi\xi}=\tilde{u}_{ij}\xi_i\xi_j$, $\tilde{u}_{ij}=u_{ij}-A_{ij}(x,u,Du)$ and $\xi\in \mathbb{R}^n$ is a unit vector. We may assume that $G$ attains its maximum at $x_0\in \Omega$ and $\xi=\xi_0$. Without loss of generality, we may assume $\{\tilde{u}_{ij}\}$ is diagonal at $x_0$ and $\xi_0=e_1$. Then the function
\begin{equation}\label{G MAX}
G(x,\xi_0)=\eta^2(x)\tilde{u}_{11}
\end{equation}
attains its maximum at $x_0$. Denoting
\begin{equation}\label{G MAX LOGARITHM}
\tilde{G}(x):=\log G(x,\xi_0)
=2 \log \eta+\log \tilde{u}_{11},
\end{equation}
then $\tilde G(x)$ also attains its maximum at $x_0$.
At $x_0$, we have
\begin{equation}\label{LOG G}
\begin{split}
\tilde{G}_i&=2\frac{\eta_i}{\eta}+\frac{D_i\tilde{u}_{11}}{\tilde{u}_{11}}=0,\\
\tilde{G}_{ij}&=2\frac{\eta_{ij}}{\eta}-2\frac{\eta_i\eta_j}{\eta^2}+\frac{D_{ij}\tilde{u}_{11}}{\tilde{u}_{11}}
-\frac{D_i\tilde{u}_{11}D_j\tilde{u}_{11}}{\tilde{u}_{11}^2}\\
&=2\frac{\eta_{ij}}{\eta}-6\frac{\eta_i\eta_j}{\eta^2}+\frac{D_{ij}\tilde{u}_{11}}{\tilde{u}_{11}},\\
\end{split}
\end{equation}
for $i,j=1,\cdots, n$, and the matrix $\{\tilde G_{ij}\} \le 0$.
From now on, we assume all the calculations are taken at $x_0$. 
Then it follows from $\{\tilde{u}^{ij} \}\geq 0$, $\tilde{u}_{11} \geq 0$ and the first equality of \eqref{LOG G} that
\begin{equation}\label{geq 0}
\begin{split}
0 &\geq \tilde{u}_{11}L\tilde{G}=\tilde{u}_{11}\sum_{i,j=1}^n\tilde{u}^{ij}D_{ij}\tilde{G}\\
&=\tilde{u}_{11}\sum_{i,j=1}^n\tilde{u}^{ij}\left [2\frac{\eta_{ij}}{\eta}-6\frac{\eta_i\eta_j}{\eta^2}+\frac{D_{ij}\tilde{u}_{11}}{\tilde{u}_{11}}\right ]\\
&\geq -C\frac{\tilde{u}_{11}}{\eta^2}\sum_{i=1}^n\tilde{u}^{ii}+\sum_{i,j=1}^n\tilde{u}^{ij}D_{ij}\tilde{u}_{11},
\end{split}
\end{equation}
where $L$ is the linearized operator defined in \eqref{linearized op}.
Recalling that $\tilde{u}_{11}=u_{11}-A_{11}$, we obtain
\begin{equation}\label{u11ij}
\begin{split}
    & \sum_{i,j=1}^n\tilde u^{ij}D_{ij}\tilde{u}_{11}=\sum_{i,j=1}^n\tilde u^{ij}D_{ij}(u_{11}-A_{11})\\
    \ge& \sum_{i,j,k,l=1}^n\tilde u^{ij} [u_{11ij} -(D_{p_k}A_{11})u_{kij} -(D_{p_kp_l}A_{11})u_{ki}u_{lj}] - C\left (1+\sum_{i=1}^n\tilde u^{ii}\right ),
\end{split}
\end{equation}
where $C$ is a constant depending on $A$ and $\sup_\Omega|Du|$.
By a direct computation, we have
\begin{equation}\label{uij11}
\begin{split}
     &\sum_{i,j=1}^n\tilde u^{ij} u_{11ij} = \sum_{i,j=1}^n\tilde u^{ij} D_{11}u_{ij} = \sum_{i,j=1}^n\tilde u^{ij} D_{11}(\tilde u_{ij} + A_{ij})  \\
 \ge & \sum_{i,j,k,l=1}^n \tilde u^{ij} [D_{11}\tilde{u}_{ij}+(D_{p_k}A_{ij})u_{k11} +(D_{p_kp_l}A_{ij})u_{k1}u_{l1}] -C\left (1+\sum_{i=1}^n\tilde u^{ii}\right ),
\end{split}
\end{equation}
where $C$ is a constant depending on $A$ and $\sup_\Omega|Du|$. By differentiating equation \eqref{EQUATION'} in the direction $\xi\in \mathbb{R}^n$ once and twice, we get
\begin{equation}\label{once dif}
\tilde u^{ij} D_\xi \tilde u_{ij} = D_\xi \tilde B,
\end{equation}
and 
\begin{equation}\label{twice dif}
\tilde u^{ij} D_{\xi\xi} \tilde u_{ij} \ge  D_{\xi\xi} \tilde B.
\end{equation}
Here the inequality \eqref{twice dif} is obtained by using the concavity of ``$\log \det$''. 
Inserting \eqref{u11ij} and \eqref{uij11} into \eqref{geq 0}, we have
\begin{equation}\label{0 ge ...}
  0 \ge  -C\frac{\tilde{u}_{11}}{\eta^2}\sum_{i=1}^n\tilde{u}^{ii}+D_{11}\tilde{B}-CD_1\tilde{B} +\sum_{i,j,k,l=1}^n\tilde{u}^{ij}[(D_{p_kp_l}A_{ij})u_{k1}u_{l1}-(D_{p_kp_l}A_{11})u_{ki}u_{lj}],
\end{equation}
where \eqref{once dif}, \eqref{twice dif} and the first equality in \eqref{LOG G} are used to deal with the terms $\sum_{i,j,k=1}^n\tilde{u}^{ij}(D_{p_k}A_{11})u_{kij}$, $\sum_{i,j=1}^n \tilde u^{ij} D_{11}\tilde{u}_{ij}$ and $\sum_{i,j,k=1}^n\tilde{u}^{ij}(D_{p_k}A_{ij})u_{k11}$ respectively, the terms $- C\left (1+\sum_{i=1}^n\tilde u^{ii}\right )$ in \eqref{u11ij} and \eqref{uij11} are absorbed in the first term on the right hand side of \eqref{0 ge ...} since we can always assume $\tilde u_{11}$ and $\sum_{i=1}^n \tilde u^{ii}$ as large as we want.
Next, we estimate the last term in \eqref{0 ge ...}. Since both $\{\tilde{u}^{ij}\}$ and $\{\tilde{u}_{ij}\}$ are diagonal at $x_0$, we get
\begin{equation}
\begin{split}
 &\sum_{i,j,k,l=1}^n\tilde{u}^{ij}[(D_{p_kp_l}A_{ij})u_{k1}u_{l1}-(D_{p_kp_l}A_{11})u_{ki}u_{lj}]\\
=&\sum_{i\neq 1}\sum_{k,l=1}^n\tilde{u}^{ii}[(D_{p_kp_l}A_{ii})u_{k1}u_{l1}-(D_{p_kp_l}A_{11})u_{ki}u_{li}]\\
=&\sum_{i\neq 1}\sum_{k,l=1}^n\tilde{u}^{ii}[D_{p_kp_l}A_{ii}(\tilde{u}_{k1}+A_{k1})(\tilde{u}_{l1}+A_{l1})-D_{p_kp_l}A_{11}(\tilde{u}_{ki}+A_{ki})(\tilde{u}_{li}+A_{li})]\\
\ge & \sum_{i\neq 1}\tilde{u}^{ii}(D_{p_1p_1}A_{ii})\tilde{u}_{11}^2-C\sum_{i=1}^n\tilde{u}^{ii}\tilde{u}_{11}.\\
\end{split}
\end{equation}
Using \eqref{D_itildeB} and \eqref{D_ijtildeB} in Corollary \ref{Coro2.1},  then \eqref{0 ge ...} becomes
\begin{equation}\label{LG1}
\begin{split}
0&\geq \sum_{i\neq 1}\tilde{u}^{ii}(D_{p_1p_1}A_{ii})\tilde{u}_{11}^2
-C\frac{\tilde{u}_{11}}{\eta^2}\sum_{i=1}^n\tilde{u}^{ii}+D_{11}\tilde{B}-CD_{1}\tilde{B}\\
& \geq \sum_{i\neq 1}\tilde{u}^{ii}(D_{p_1p_1}A_{ii})\tilde{u}_{11}^2
-C\frac{\tilde{u}_{11}}{\eta^2}\sum_{i=1}^n\tilde{u}^{ii}-C\tilde{u}_{11}B^{-\frac{1}{n-1}}-C\tilde{u}_{11}^2 + \sum_{k=1}^n \tilde B_{p_k}D_{11}u_k\\
& \geq \sum_{i\neq 1}\tilde{u}^{ii}(D_{p_1p_1}A_{ii})\tilde{u}_{11}^2
-C\frac{\tilde{u}_{11}}{\eta^2}\sum_{i=1}^n\tilde{u}^{ii}-C\frac{\tilde{u}_{11}}{\eta}B^{-\frac{1}{n-1}}-C\tilde{u}_{11}^2,
\end{split}
\end{equation}
where the third order derivative term $\sum_{k=1}^n \tilde B_{p_k}D_{11}u_k$ is treated by using \eqref{B_i over B'} and the first equality in \eqref{LOG G}. 
Note that the constant $C$ changes from line to line in the context.
Since $\tilde{u}_{11}\ge 1$, we can get
\begin{equation}\label{Buii}
\begin{split}
\sum_{i=1}^n\tilde{u}^{ii}&\geq \sum_{i=2}^n\tilde{u}^{ii}
\geq (n-1)\left (\prod_{i=2}^n\tilde{u}^{ii}\right )^{\frac{1}{n-1}}\\
&=(n-1)\left (\prod_{i=1}^n\tilde{u}^{ii}\right )^{\frac{1}{n-1}}(\tilde{u}_{11})^{\frac{1}{n-1}}\\
&\geq(n-1)B^{-\frac{1}{n-1}}.
\end{split}
\end{equation}
Plugging \eqref{Buii} into \eqref{LG1}, we obtain
\begin{equation}\label{LG}
0\geq \sum_{i\neq 1}\tilde{u}^{ii}D_{p_1p_1}A_{ii}\tilde{u}_{11}^2
-C\frac{\tilde{u}_{11}}{\eta^2}\sum_{i=1}^n\tilde{u}^{ii}-C(\tilde{u}_{11})^2.
\end{equation}
By the A3 condition, choosing $\tilde{\xi}=\tilde{u}_{11}e_1$ and $\tilde{\eta}=\sum\limits_{i=2}^n\sqrt{\tilde{u}^{ii}}e_i$, we have
\begin{equation}\label{using A3}
\begin{split}
\sum_{i\neq 1}\tilde{u}^{ii}(D_{p_1p_1}A_{ii})\tilde{u}_{11}^2 &= \sum_{i,j,k,l=1}^n D^{2}_{p_kp_l}A_{ij}\tilde\xi_i\tilde \xi_j\tilde \eta_k\tilde \eta_l \\
&\geq c_0\tilde{u}_{11}^2\sum_{i=2}^n\tilde{u}^{ii} \\
&\geq \frac{1}{n}c_0\tilde{u}_{11}^2\sum_{i=2}^n\tilde{u}^{ii}+\frac{n-1}{n}c_0\tilde{u}_{11}^2\tilde{u}^{22}\\
&\geq \frac{1}{n}c_0\tilde{u}_{11}^2\sum_{i=2}^n\tilde{u}^{ii}+\frac{n-1}{n}c_0\tilde{u}_{11}^2\tilde{u}^{11}\\
&\geq \frac{1}{n}c_0\tilde{u}_{11}^2\sum_{i=1}^n\tilde{u}^{ii}.\\
\end{split}
\end{equation}
Without loss of generality, we assume 
\begin{equation}\label{wlg}
\frac{1}{2n}c_0\sum_{i=1}^n\tilde{u}^{ii}\geq C. 
\end{equation}
Otherwise we are done. Combining \eqref{LG}, \eqref{using A3} and \eqref{wlg}, we have
\begin{equation}
0 \geq \frac{1}{2n}c_0\tilde{u}_{11}^2\sum_{i=1}^n\tilde{u}^{ii}-C\frac{\tilde{u}_{11}}{\eta^2}\sum_{i=1}^n\tilde{u}^{ii},
\end{equation}
which leads to
\begin{equation}\label{estimates}
\eta^2\tilde{u}_{11} \leq C.
\end{equation}
We now complete the proof of Theorem \ref{Th1.1}.
\end{proof}

Note that the constant $C$ in \eqref{second order derivatives estimates} in Theorem \ref{Th1.1} is independent of the positive lower bound of $B$. Then the $C^{1,1}$ regularity result under the A3 condition, Theorem \ref{Th1.2}, follows directly from the interior estimates in Theorem \ref{Th1.1}. Here we omit the proof of Theorem \ref{Th1.2} since it is standard.

\vspace{3mm}


\section{Interior regularity for the DMATE \eqref{EQUATION} under the A3w$^+$ condition}\label{Section 4}

In this section, we prove the Pogorelov type estimate in Theorem \ref{Th1.3} under the A3w$^+$ condition and suitable barrier conditions, which can be applied to the interior $C^{1,1}$ regularity for solutions of the DMATE \eqref{EQUATION} in Theorem \ref{Th1.4}.

\begin{proof}[Proof of Theorem \ref{Th1.3}.]
First we note that under either (i) or (ii), we have
\begin{equation}\label{both barriers}
L \varphi \ge \varepsilon_1 \sum_{i=1}^n \tilde u^{ii} - C,
\end{equation}
for some positive constants $\varepsilon_1$ and $C$. In case (i), $\varphi$ is the function in the $A$-boundedness condition \eqref{A boundedness}, and \eqref{both barriers} with $\varepsilon_1=1$ can be calculated directly from \eqref{A boundedness}. While in case (ii), the inequality \eqref{both barriers} with $\varphi=e^{\kappa(\underline u-u)}$ is proved in \eqref{barrier inequality} in Lemma \ref{Lemma 2}.

We construct the auxiliary function
\begin{equation}\label{barrier function h}
h(x, \xi)=\eta^{\alpha}\tilde{u}_{\xi\xi}e^{\frac{1}{2}\beta|Du|^2+\gamma\varphi},
\end{equation}
where $\varphi$ is the barrier function in \eqref{both barriers}, $\tilde{u}_{\xi\xi}=\tilde{u}_{ij}\xi_i\xi_j$, $\xi=(\xi_1, \cdots, \xi_n)$ and $|\xi|=1$, $\tilde{u}_{ij}=u_{ij}-A_{ij}$, $\eta= w-u$ and $\alpha, \beta, \gamma$ are positive constants to be determined.

Since $h \geq 0$ in $\Omega$ and $h=0$ on $\partial\Omega$, we may assume that $h$ attains its maximum at the point $\bar{x}\in \Omega$ and some unit vector $\bar{\xi}$. We may assume $u(\bar{x}) < w(\bar{x})$, namely $\eta(\bar x)>0$. By taking the logarithm of $h$, we obtain
\begin{equation}\label{barrier function bar h}
\bar{h}(x,\xi):= \log h(x, \xi)=\alpha\log \eta+\log(\tilde{u}_{ij}\xi_i\xi_j)+\frac{1}{2}\beta|Du|^2+\gamma \varphi.
\end{equation}
Thus, $\bar{h}$ also attains its maximum at the point $\bar{x}\in \Omega$ and the vector $\bar{\xi}$.
We may assume that $\bar{\xi}=(1,0,\cdots,0)$ and $\{\tilde{u}_{ij}\}$ is diagonal at $\bar{x}$. We define
\begin{equation}\label{barrier function}
\begin{split}
v(x)&:=\bar{h}(x,\xi)|_{\xi=\bar{\xi}}\\
&=\alpha\log \eta+\log(\tilde{u}_{11})+\frac{1}{2}\beta|Du|^2+\gamma \varphi.
\end{split}
\end{equation}
Since $\bar{x}$ is also the maximum point of $v$, we have
\begin{equation}\label{Dv}
 Dv(\bar{x})=0,
\end{equation}
and
\begin{equation} \label{D2v}
D^2v(\bar{x})\leq 0.
\end{equation}
It follows from \eqref{Dv}, \eqref{D2v} and $\{\tilde{u}^{ij}\} \geq 0$ that
\begin{equation}\label{Lv0}
Lv(\bar{x})\leq 0,
\end{equation}
where $L$ is the linearized operator defined in \eqref{linearized op}. By a direct computation, we have, at $\bar{x}$,
\begin{equation}\label{first derivative of v}
D_iv=\frac{\alpha D_i\eta}{\eta}+\frac{D_i\tilde{u}_{11}}{\tilde{u}_{11}}+\beta D_kuD_{ki}u+\gamma D_i\varphi,
\end{equation}
and
\begin{equation}\label{second derivative of v}
\begin{split}
D_{ii}v=&\frac{\alpha D_{ii}\eta}{\eta}-\frac{\alpha(D_i\eta)^2}{\eta^2}+\frac{D_{ii}\tilde{u}_{11}}{\tilde{u}_{11}}-\frac{(D_{i}\tilde{u}_{11})^2}{\tilde{u}_{11}^2}\\
&+\beta\sum_{i,k=1}^{n}\left ((D_{ik}u)^2+(D_ku)D_{iik}u\right )+\gamma D_{ii} \varphi,
\end{split}
\end{equation}
for $i=1,\cdots,n$.
Inserting \eqref{first derivative of v} and \eqref{second derivative of v} into \eqref{Lv0}, we get
\begin{equation}\label{expression of Lv}
\begin{split}
0\ge & Lv(\bar{x})\\
=&\frac{\alpha}{\eta}L\eta-\frac{\alpha}{\eta^2}\sum_{i=1}^n\tilde{u}^{ii}(D_i\eta)^2+\frac{1}{\tilde{u}_{11}}L\tilde{u}_{11}
-\frac{1}{\tilde{u}_{11}^2}\sum_{i=1}^n\tilde u^{ii}(D_i\tilde{u}_{11})^2\\
&+\beta\sum_{k=1}^{n}D_kuLu_k+\beta\sum_{i, k=1}^{n}\tilde{u}^{ii}(D_{ik}u)^2+\gamma L\varphi.
\end{split}
\end{equation}

Next, we estimate each term of \eqref{expression of Lv}. From now on, all calculations are made at the maximum point $\bar{x}$. We first consider the general case that $B$ depends on $p$, namely $B_p\not \equiv 0$. By calculations, we have
\begin{equation}\label{Leta}
\begin{split}
L\eta
= &\sum_{i=1}^n\tilde{u}^{ii}\left [D_{ii}w-\tilde{u}_{ii}-A_{ii}(x, u, Du)-\sum_{k=1}^n(D_{p_k}A_{ii}(x, u, Du))D_k\eta\right ]- \sum_{k=1}^n\tilde B_{p_k}D_k\eta\\
\ge & -n+\sum_{i=1}^n\tilde{u}^{ii}\left [A_{ii}(x, u, Dw)-A_{ii}(x, u, Du)-\sum_{k=1}^n(D_{p_k}A_{ii}(x, u, Du))D_k\eta\right ]- \sum_{k=1}^n\tilde B_{p_k}D_k\eta\\
\ge & -n-CB^{-\frac{1}{2(n-1)}}+\frac{1}{2}\sum_{i,k,l=1}^n\tilde{u}^{ii}A_{ii, kl}(x,u,\bar{p})D_k\eta D_l\eta\\
\ge & -n-CB^{-\frac{1}{2(n-1)}}-\frac{1}{2}\mu_0^-\sum_{i=1}^n \tilde{u}^{ii}(D_i\eta)^2,
\end{split}
\end{equation}
for $\bar{p}=(1-\theta)Du+\theta Dw$ and $\theta\in(0,1)$, where $D_{ii}w-A_{ii}(x, u, Dw)\geq 0$ is used to obtain the first inequality, Taylor's formula and \eqref{B_i over B} are used to obtain the second inequality, the A3w$^+$ condition is used to obtain the third inequality, $\mu_0^-=-\min \{\mu_0, 0\}$ and $\mu_0$ is the constant in \eqref{A3w+}.
Using the Cauchy's inequality, it follows from \eqref{Leta} that
\begin{equation}\label{important eta est}
\begin{split}
\frac{\alpha}{\eta}L\eta
\ge & -\alpha \left [\frac{n}{\eta}+\frac{C}{\eta}B^{-\frac{1}{2(n-1)}}+\frac{\mu_0^-}{2\eta}\sum_{i=1}^n\tilde{u}^{ii}(D_i\eta)^2 \right ] \\
\ge & -  \frac{ \alpha n}{\eta}-\frac{\alpha^2C}{\eta^2} - C B^{-\frac{1}{n-1}} - \frac{\alpha \mu_0^-}{2} \sum_{i=1}^n\tilde{u}^{ii}\frac{(D_i\eta)^2}{\eta^2},
\end{split}
\end{equation}
where we have assumed $\eta(\bar x)\in (0,1]$. We will show the trivial case when $\eta(\bar x)>1$ at the end of the proof.

In order to estimate $\frac{1}{\tilde{u}_{11}}L\tilde{u}_{11}$, we first calculate $Lu_{11}.$ We can assume $\tilde{u}_{11} \geq 1$, otherwise we are done. By a direct computation and using \eqref{Differentiation of the equation twice} with $\xi=e_1$, we have
\begin{equation}\label{Lu11}
\begin{split}
Lu_{11}&\ge \sum_{i=1}^n\tilde{u}^{ii}\tilde{u}^{jj}(D_1\tilde{u}_{ij})^2+\sum_{i, k, l=1}^n\tilde{u}^{ii}A_{ii, kl}u_{k1}u_{l1}
+D_{11}\tilde{B} - \tilde B_{p_k}D_ku_{11}-C\sum_{i, j=1}^n[(1+\tilde{u}_{jj})\tilde{u}^{ii}]\\
&\ge \sum_{i=1}^n\tilde{u}^{ii}\tilde{u}^{jj}(D_1\tilde{u}_{ij})^2-C\sum_{i=1}^n\tilde{u}_{ii}
+D_{11}\tilde{B} - \tilde B_{p_k}D_ku_{11}-C\sum_{i, j=1}^n[(1+\tilde{u}_{jj})\tilde{u}^{ii}],
\end{split}
\end{equation}
where the A3w$^+$ condition is used to obtain the second inequality. With the help of  \eqref{D_ijtildeB} in Corollary \ref{Coro2.1}, we can further get
\begin{equation}\label{result of Lu11}
\begin{split}
Lu_{11}\geq & \sum_{i, j=1}^n\tilde{u}^{ii}\tilde{u}^{jj}(D_1\tilde{u}_{ij})^2-C\sum_{i=1}^n \tilde{u}_{ii}-C\sum_{i, j=1}^n[(1+\tilde{u}_{jj})\tilde{u}^{ii}] \\
                    & -C(1+\tilde{u}_{11}) B^{-\frac{1}{n-1}}-C(1+\tilde{u}_{11})^2\\
\geq & \sum_{i, j=1}^n\tilde{u}^{ii}\tilde{u}^{jj}(D_1\tilde{u}_{ij})^2-C\sum_{i, j=1}^n\tilde{u}_{jj}\tilde{u}^{ii} -C\tilde{u}_{11}B^{-\frac{1}{n-1}}-C\tilde{u}_{11}^2,
\end{split}
\end{equation}
where we assume $\tilde u_{11}\ge 1$ and $\sum\limits_{i=1}^n \tilde u^{ii} \ge 1$ to obtain the second inequality.
Note that the third order term $-\tilde B_{p_k}D_ku_{11}$ in \eqref{Lu11} is eliminated by the last term of \eqref{D_ijtildeB}.
Next, we calculate $LA_{11}$. Using the definition of $L$, $\tilde{u}_{ij}=u_{ij}-A_{ij}$ and the $C^2$ smoothness of $A$, we obtain
\begin{equation}\label{LA11}
\begin{split}
LA_{11}&\leq C+C\sum_{i, j=1}^n[(1+\tilde{u}_{jj})\tilde{u}^{ii}]+\sum_{i, j, k, l=1}^n\tilde{u}^{ij}D_{p_kp_l}A_{11}\tilde{u}_{ki}\tilde{u}_{kj}+D_k\tilde{B}\\
&\leq C\sum_{i, j=1}^n[(1+\tilde{u}_{jj})\tilde{u}^{ii}+\tilde{u}_{jj}]+D_k\tilde{B}\\
&\leq C\sum_{i, j=1}^n[(1+\tilde{u}_{jj})\tilde{u}^{ii}]+D_k\tilde{B}\\
&\leq C\sum_{i, j=1}^n\tilde{u}_{jj}\tilde{u}^{ii}+C\tilde{u}_{11}B^{-\frac{1}{n-1}},
\end{split}
\end{equation}
where we again assume $\tilde u_{11}\ge 1$ and $\sum\limits_{i=1}^n \tilde u^{ii} \ge 1$.
Recalling $\tilde u_{11}=u_{11}-A_{11}$, we get from \eqref{result of Lu11} and \eqref{LA11} that
\begin{equation}\label{3.19}
L\tilde{u}_{11}\ge \sum_{i, j=1}^n\tilde{u}^{ii}\tilde{u}^{jj}(D_1\tilde{u}_{ij})^2-C\sum_{i, j=1}^n\tilde{u}_{jj}\tilde{u}^{ii}-C\tilde{u}_{11}B^{-\frac{1}{n-1}}-C\tilde{u}_{11}^2.
\end{equation}
Therefore, we have
\begin{equation}\label{3.19a}
\frac{1}{\tilde{u}_{11}}L\tilde{u}_{11}
\geq\frac{1}{\tilde{u}_{11}}\sum_{i, j=1}^n\tilde{u}^{ii}\tilde{u}^{jj}(D_1\tilde{u}_{ij})^2-C\sum_{i=1}^n (\tilde u^{ii}+\tilde u_{ii})-CB^{-\frac{1}{n-1}}.
\end{equation}
Choosing $\xi=e_k$ in \eqref{Differentiation of the equation once}, we have
\begin{equation}\label{Luk}
\begin{split}
Lu_k&=\sum_{i=1}^n\tilde{u}^{ii}\left [D_{ii}u_{k}-\sum_{l=1}^n(D_{p_l}A_{ii})D_lu_k\right ]-\sum_{l=1}^n\frac{B_{p_l}}{B}u_{lk}\\
&=\sum_{i=1}^n\tilde{u}^{ii}D_kA_{ii}+\sum_{i=1}^n\tilde{u}^{ii}(D_uA_{ij})u_k+\frac{B_k}{B}+\frac{B_z}{B}u_k,
\end{split}
\end{equation}
for $k=1,\cdots, n$. Hence, we have
\begin{equation}\label{3.21}
\beta \sum_{k=1}^nD_kuLu_k \geq -\beta C\sum_{i=1}^n\tilde{u}^{ii}-\beta CB^{-\frac{1}{n-1}}.
\end{equation}
By a direct calculation, we have
\begin{equation}\label{3.22}
\begin{split}
\beta \sum_{i, k=1}^n\tilde{u}^{ii}(D_{ik}u)^2&=\beta\sum_{i=1}^n\tilde{u}^{ii}(\tilde{u}_{ii}+A_{ii})^2+\beta\sum_{k\neq i}\tilde{u}^{ii}A_{ik}^2\\
&\geq \beta \sum_{i=1}^n\tilde{u}_{ii}-\beta C\sum_{i=1}^n\tilde{u}^{ii}.
\end{split}
\end{equation}
From the barrier inequality \eqref{both barriers} in both cases (i) and (ii), we can also have
\begin{equation}\label{both barriers'}
\gamma L \varphi \ge \frac{1}{2}\varepsilon_1 \gamma \sum_{i=1}^n \tilde u^{ii},
\end{equation}
by assuming $\sum\limits_{i=1}^n\tilde{u}^{ii}\geq \frac{2C}{\varepsilon_1}$. Now choosing $\alpha\ge 1$ and $\beta \ge 1$ and inserting \eqref{important eta est}, \eqref{3.19a}, \eqref{3.21}, \eqref{3.22} and \eqref{both barriers'} into \eqref{expression of Lv}, we obtain
\begin{equation}\label{3.27}
\begin{split}
0 \geq & - \frac{\alpha^2C}{\eta^2} - \beta C B^{-\frac{1}{n-1}} + \left (\frac{1}{2}\gamma\varepsilon_1-\beta C\right )\sum_{i=1}^n\tilde{u}^{ii}+(\beta-C)\sum_{i=1}^n\tilde{u}_{ii}\\
&-\alpha C\sum_{i=1}^n\tilde{u}^{ii}\frac{(D_i\eta)^2}{\eta^2}+\frac{1}{\tilde{u}_{11}}\sum_{i,j=1}^n\tilde{u}^{ii}\tilde{u}^{jj}(D_1\tilde{u}_{ij})^2-\frac{1}{\tilde{u}_{11}^2}\sum_{i=1}^n\tilde{u}^{ii}(D_i\tilde{u}_{11})^2.
\end{split}
\end{equation}
Splitting $\sum_{i=1}^n\tilde{u}^{ii}\frac{(D_i\eta)^2}{\eta^2}$ into two parts, we have
\begin{equation}\label{splitting}
\sum_{i=1}^n\tilde{u}^{ii}\frac{(D_i\eta)^2}{\eta^2}=\frac{(D_1\eta)^2}{\eta^2\tilde{u}_{11}}+\sum_{i=2}^n\tilde{u}^{ii}\frac{(D_i\eta)^2}{\eta^2}.
\end{equation}
Observing that the first term on the right hand side of \eqref{splitting} can be absorbed by the first term on the right hand side of \eqref{3.27}, we only need to estimate the last term in \eqref{splitting}. From \eqref{Dv} and \eqref{first derivative of v}, we have
\begin{equation}\label{derivative of eta}
\begin{split}
\alpha C\sum_{i=2}^n\tilde{u}^{ii} \frac{(D_i\eta)^2}{\eta^2}&=\alpha C\sum_{i=2}^n\tilde{u}^{ii} \left \{\frac{1}{\alpha^2}\left [\frac{D_i\tilde{u}_{11}}{\tilde{u}_{11}}+\beta D_ku(\tilde{u}_{ki}-A_{ki})+\gamma D_i\varphi  \right ]^2\right \} \\
&\leq \frac{C}{\alpha} \sum_{i=2}^n\tilde{u}^{ii} \left \{\left (\frac{D_i\tilde{u}_{11}}{\tilde{u}_{11}}\right )^2+\beta^2(\tilde{u}_{ii}^2+1)+\gamma^2 ( D_i\varphi )^2\right \} \\
&\leq \frac{1}{2\tilde{u}_{11}^2}\sum_{i=2}^n \tilde{u}^{ii}  (D_i\tilde{u}_{11})^2+\sum_{i=1}^n(\tilde{u}_{ii}+\tilde{u}^{ii}),
\end{split}
\end{equation}
where we choose $\alpha=(\beta^2+\gamma^2+2)C$. Thus, from \eqref{3.27}, \eqref{splitting} and \eqref{derivative of eta}, we have
\begin{equation}\label{before Pog}
\begin{split}
0 \geq & - \frac{\alpha^2C}{\eta^2} - \beta C B^{-\frac{1}{n-1}} +\left (\frac{1}{2}\gamma\varepsilon_1-\beta C\right )\sum_{i=1}^n\tilde{u}^{ii}+(\beta-C)\sum_{i=1}^n\tilde{u}_{ii}\\
&-\frac{1}{2\tilde{u}_{11}^2}\sum_{i=2}^n \tilde{u}^{ii}  (D_i\tilde{u}_{11})^2+\frac{1}{\tilde{u}_{11}}\sum_{i,j=1}^n\tilde{u}^{ii}\tilde{u}^{jj}(D_1\tilde{u}_{ij})^2-\frac{1}{\tilde{u}_{11}^2}\sum_{i=1}^n\tilde{u}^{ii}(D_i\tilde{u}_{11})^2.
\end{split}
\end{equation}
Using the Pogorelov term $\frac{1}{\tilde{u}_{11}}\sum_{i,j=1}^n\tilde{u}^{ii}\tilde{u}^{jj}(D_1\tilde{u}_{ij})^2$, we have
\begin{equation}\label{3.28}
\begin{split}
&-\frac{1}{2\tilde{u}_{11}^2}\sum_{i=2}^n \tilde{u}^{ii}  (D_i\tilde{u}_{11})^2+\frac{1}{\tilde{u}_{11}}\sum_{i,j=1}^n\tilde{u}^{ii}\tilde{u}^{jj}(D_1\tilde{u}_{ij})^2-\frac{1}{\tilde{u}_{11}^2}\sum_{i=1}^n\tilde{u}^{ii}(D_i\tilde{u}_{11})^2 \\
= &\sum_{i,j=1}^n\frac{1}{\tilde{u}_{11}}\tilde{u}^{ii}\tilde{u}^{jj}(D_1\tilde{u}_{ij})^2-\frac{1}{\tilde{u}_{11}^2}\tilde{u}^{11}(D_1\tilde{u}_{11})^2
-\frac{3}{2\tilde{u}_{11}^2}\sum_{i=2}^n\tilde{u}^{ii}(D_i\tilde{u}_{11})^2\\
\ge &\frac{1}{2\tilde{u}_{11}^2}\sum_{i=2}^n\tilde{u}^{ii}(D_i\tilde{u}_{11})^2+\frac{2}{\tilde{u}_{11}^2}\sum_{i=2}^n\tilde{u}^{ii}[(D_1\tilde{u}_{1i})^2
-(D_i\tilde{u}_{11})^2]\\
\ge &\frac{1}{2\tilde{u}_{11}^2}\sum_{i=2}^n\tilde{u}^{ii}(D_i\tilde{u}_{11})^2
+\frac{2}{\tilde{u}_{11}^2}\sum_{i=2}^n\tilde{u}^{ii}(D_iA_{11}-D_1A_{1i})(2D_i\tilde{u}_{11}+D_iA_{11}-D_1A_{1i})\\
\ge & -\frac{C}{\tilde{u}_{11}^2}\sum_{i=1}^n\tilde{u}^{ii} \ge -C\sum_{i=1}^n\tilde{u}^{ii},
\end{split}
\end{equation}
where Cauchy's inequality is used in the second last inequality.
Therefore, from \eqref{before Pog} and \eqref{3.28}, we have
\begin{equation}\label{nearly final}
0 \geq  - \frac{\alpha^2C}{\eta^2} - \beta C B^{-\frac{1}{n-1}} + \left (\frac{1}{2}\gamma\varepsilon_1-\beta C\right )\sum_{i=1}^n\tilde{u}^{ii}+(\beta-C)\sum_{i=1}^n\tilde{u}_{ii}.
\end{equation}
By using the key relationship \eqref{Buii} between $B^{-\frac{1}{n-1}}$ and $\sum_{i=1}^n\tilde{u}^{ii}$, we have from  \eqref{nearly final} that
\begin{equation}\label{final}
0 \geq  - \frac{\alpha^2C}{\eta^2} + \left (\frac{1}{2}\gamma\varepsilon_1-\beta C\right )\sum_{i=1}^n\tilde{u}^{ii}+(\beta-C)\sum_{i=1}^n\tilde{u}_{ii}.
\end{equation}
By choosing $\beta=C+1$ and $\gamma=\frac{2\beta C}{\varepsilon_1}$, \eqref{final} becomes
\begin{equation}\label{final'}
0 \geq  - \frac{\alpha^2C}{\eta^2} + \sum_{i=1}^n\tilde{u}_{ii} \geq  - \frac{\alpha^2C}{\eta^2} + \tilde{u}_{11},
\end{equation}
which leads to
\begin{equation}\label{3.30}
\eta^2\tilde{u}_{11}(\bar{x})\leq \alpha^2 C.
\end{equation}
We then immediately get the conclusion \eqref{pOGORELOV ESTIMATES} in the $B_p \not \equiv 0$ case.

For the case $B_p \equiv 0$ or more general case $\frac{|B_p|}{B}\le C$, \eqref{important eta est} can be replaced by 
\begin{equation}
\frac{\alpha}{\eta}L\eta
\ge -  \frac{ \alpha C}{\eta}- C^\prime B^{-\frac{1}{n-1}} - \frac{\alpha \mu_0}{2} \sum_{i=1}^n\tilde{u}^{ii}\frac{(D_i\eta)^2}{\eta^2},
\end{equation}
(the constant $C^\prime=0$ when $B_p \equiv 0$), and correspondingly, \eqref{final'} can be replaced by
\begin{equation}\label{final''}
0 \geq  - \frac{\alpha C}{\eta} + \tilde{u}_{11},
\end{equation}
which leads to
\begin{equation}\label{3.30}
\eta \tilde{u}_{11}(\bar{x})\leq \alpha C.
\end{equation}
We then immediately get the conclusion \eqref{pOGORELOV ESTIMATES} in the $B_p \equiv 0$ case or more general $\frac{|B_p|}{B}\le C$ case.
Then Theorem \ref{Th1.3} is proved provided $\eta(\bar{x}) \in (0,1]$. 

While if $\eta(\bar{x}) > 1$, \eqref{important eta est} still holds. Furthermore, $\eta$ in the denominators on the right hand side of \eqref{important eta est} can be replaced by $1$. Following the above proof, we can have
\begin{equation}\label{trivial}
\tilde{u}_{11}(\bar{x})\leq C,
\end{equation}
which also leads to the conclusion \eqref{pOGORELOV ESTIMATES}.

We now complete the proof of Theorem \ref{Th1.3}.
\end{proof}

\begin{remark}\label{Rm4.1}
In the above proof, the A3w$^+$ condition is crucial in the critical inequality \eqref{Leta}, which is the reason why we restrict our study in the class of $A$ satisfying A3w$^+$. Alternative conditions to get through the inequality \eqref{Leta} can be found in (2.4), Remark 2.1 and Remark 2.2 in \cite{LTW}.
Note that the inequality \eqref{Lu11}, which is deduced from the A3w$^+$ condition, can also be derived by just using the A3w condition and some other conditions, see \cite{LT-Pog, LTW}.
\end{remark}

We are now ready to prove Theorem \ref{Th1.4}.

\begin{proof}[The proof of Theorem \ref{Th1.4}.]
Let $\Omega_j$ be a sequence of $C^{\infty}$ bounded domains such that $\Omega_j\rightarrow \Omega$ as $j\rightarrow\infty$. 
Note that if in case (i), these domains also need to satisfy the $A$-boundedness condition.
We can find $B_j\in C^{\infty}$ such that $B_j > 0$, $B_j$ tends uniformly to $B$ in $\Omega$ and $\|B_j^{\frac{1}{n-1}}\|_{C^{1,1}(\bar{\Omega}_j \times \mathbb{R} \times \mathbb{R}^n)} \leqslant C$ for some uniform constant $C$, (independent of $j$). From the existence result in \cite{JTY}, the Dirichlet problem $\det (M[u_j])=B_j$ in $\Omega_j$, $u_j=w$ on $\partial\Omega_j$, has a unique classical solution $u_j\in C^3(\bar \Omega_j)$.

Since $A$ and $B$ are nondecreasing in $z$, from the strong maximum principle, either $u\equiv w$ in $\Omega$ or $u<w$ in $\Omega$. In the former case, since $w\in C^{1,1}(\Omega)$, we immediately have $u\in C^{1,1}(\Omega)$. Next, we only consider the latter case when $u<w$ in $\Omega$.
Since $u_j$ is a degenerate elliptic solution, we can have the uniform gradient estimate from \cite{JXX}.
By applying the Pogorelov type estimate \eqref{pOGORELOV ESTIMATES} in the domain $\{u_j < w-\varepsilon\}$ for any fixed small constant $\varepsilon>0$, we have
\begin{equation}
(w-u_j-\varepsilon)^\tau |D^2u_j|\leq C, \quad {\rm in} \ \{u_j < w-\varepsilon\},
\end{equation}
where the constant $C$ is independent of $j$.
Thus, we have
\begin{equation}
|D^2u_j|\leq C, \quad {\rm in} \ \{u_j <  w-2\varepsilon\},
\end{equation}
where the constant $C$ is independent of $j$.
From the stability property of viscosity solutions \cite{CIL}, we have $u_j \rightarrow u$ as $j\rightarrow \infty$, and
\begin{equation}\label{C11estimate}
u\in C^{1,1}(\{u < w-2\varepsilon\}),
\end{equation}
for any fixed small constant $\varepsilon>0$. Since the domain $\{u < w-2\varepsilon\}$ tends to $\Omega=\{u < w\}$ as $\varepsilon$ to $0$, from \eqref{C11estimate}, we finally get $u\in C^{1,1}(\Omega)$.
\end{proof}

\vspace{3mm}

\bibliographystyle{amsplain}

\end{document}